\documentclass[a4paper,10pt]{article}
\usepackage{geometry}
\usepackage[english]{babel}
\usepackage{hyperref}
\usepackage{amsmath,amssymb,stmaryrd,enumerate,latexsym,theorem,bbm}

\newcommand{\assign}{:=}
\newcommand{\comma}{{,}}
\newcommand{\mathd}{\mathrm{d}}
\newcommand{\nobracket}{}
\newcommand{\tmaffiliation}[1]{\\ #1}
\newcommand{\tmemail}[1]{\\ \textit{Email:} \texttt{#1}}
\newcommand{\tmop}[1]{\ensuremath{\operatorname{#1}}}
\newcommand{\tmtextbf}[1]{{\bfseries{#1}}}
\newcommand{\tmtextit}[1]{{\itshape{#1}}}

\newenvironment{proof}{\noindent\textbf{Proof\ }}{\hspace*{\fill}$\Box$\medskip}
\newenvironment{proof*}[1]{\noindent\textbf{#1\ }}{\hspace*{\fill}$\Box$\medskip}
\newtheorem{theorem}{Theorem}

\newtheorem{corollary}[theorem]{Corollary}
\newtheorem{lemma}[theorem]{Lemma}
{\theorembodyfont{\rmfamily}\newtheorem{remark}[theorem]{Remark}}

\newcommand{\textdots}[0]{...}

\begin{document}

\title{
  A note on supersymmetry and  stochastic differential equations
}

\author{Francesco C. De Vecchi and Massimiliano Gubinelli\\
  \tmaffiliation{
  Institute for Applied Mathematics \&\\
  Hausdorff Center for Mathematics\\
  University of Bonn, Germany}\\
  \tmemail{fdevecch@uni-bonn.de}
  \tmemail{mgubinel@uni-bonn.de}
}

\date{December 2019}

\maketitle

\begin{abstract}
  We obtain a dimensional reduction result for the law of a class of
  stochastic differential equations using a supersymmetric representation
  first introduced by Parisi and Sourlas.
\end{abstract}

\section{Introduction}

In this paper we want to exploit a supersymmetric representation of scalar
stochastic differential equations~(SDEs) with additive noise and nonlinear
drift $V'$ in order to prove the well known relation between the invariant law
of these SDEs and the Gibbs measure $e^{- 2 V (x)} \mathd x$.

The supersymmetric representation of SDEs or more generally SPDEs was first
noted by Parisi and
Sourlas~{\cite{parisi_random_1979,parisi_supersymmetric_1982}} and it is well
known and used in the physics literature (see, e.g.~{\cite{Zinn1993}}) where
the relation between supersymmetry, SDEs and Gibbs measures (called
dimensional reduction) was formally
established~{\cite{damgaard_stochastic_1988,Gozzi1984}}. In the case of
elliptic SPDEs these formal arguments have been rigorously exploited and
proved~{\cite{albeverio_elliptic_2018,Klein1984}} and applied to stochastic
quantization program of quantum field theory~{\cite{AlDeGu2019,GuHof2019}}. In
the present paper we want to propose a similar rigorous version of dimensional
reduction for one dimensional SDEs. The proof proposed here follows more
closely the methods used for dimensional reduction of elliptic equations used
in~{\cite{albeverio_elliptic_2018}} (see also~{\cite{Klein1984}}) rather then
the formal proofs of the physics literature (see,
e.g.~{\cite{damgaard_stochastic_1988,Gozzi1984}})

\

More precisely, here we consider the following SDE
\begin{equation}
  \partial_t \phi (t) + m^2 \phi (t) + f (t) V' (\phi (t)) = \xi (t)
  \label{eq:SDEmain}, \qquad t \in \mathbb{R},
\end{equation}
where $m > 0$, $f : \mathbb{R} \rightarrow \mathbb{R}_+$ is a compactly
supported positive even smooth function such that $f (0) = 1$, $V : \mathbb{R}
\rightarrow \mathbb{R}$ is a smooth bounded function with all derivatives
bounded and $\xi$ is a white noise on $\mathbb{R}$. Eq.~{\eqref{eq:SDEmain}}
has a unique solution $\phi : \mathbb{R} \rightarrow \mathbb{R}$ which
coincides for sufficiently negative times with the Ornstein--Uhlenbeck process
$\varphi = \mathcal{G} \ast \xi$ where
\[ \mathcal{G} (t) = e^{- m^2 t} \mathbb{I}_{t > 0} . \]
This solution satisfies the integral equation
\begin{equation}
  \phi (t) + \mathcal{G} \ast (f V' (\phi)) (t) = \varphi (t),
  \label{eq:main2} \qquad t \in \mathbb{R},
\end{equation}
and moreover its law is invariant under the inversion $t \mapsto - t$ of the
time variable.

\

The aim of this note is to prove the following theorem.

\begin{theorem}
  \label{theorem_main}For any bounded measurable function $F : \mathbb{R}
  \rightarrow \mathbb{R}$ we have
  \[ \mathbb{E} \left[ F (\phi (0)) e^{- 2 \int_{- \infty}^0 f' (t) V (\phi
     (t)) \mathd t} \right] = \frac{\mathcal{} 1}{\mathcal{Z}} \left.
     \int_{\mathbb{R}} F (x) e^{- m^2 x^2 - 2 V (x)} \mathd x \right. \]
  where
  \[ \mathcal{Z}= \frac{\mathbb{E} \left[ e^{- 2 \int_{- \infty}^0 f' (t) V
     (\phi (t)) \mathd t} \right]}{\int_{\mathbb{R}} e^{- m^2 x^2 - 2 V (x)}
     \mathd x} . \]
\end{theorem}

\begin{proof*}{Proof}
  Let $\mu_{\varphi}$ be the law of the Gaussian field $\varphi = \mathcal{G}
  \ast \xi$ on the space $C (\mathbb{R}; \mathbb{R})$ endowed with the
  topology of uniform convergence on bounded intervals. Girsanov theorem
  implies that for any measurable bounded function $F : \mathbb{R} \rightarrow
  \mathbb{R}$
  \begin{equation}
    \mathbb{E} \left[ F (\phi (0)) e^{- 2 \int_{- \infty}^0 f' (t) V (\phi
    (t)) \mathd t} \right] = \int F (\varphi (0)) \exp (S (\varphi))
    \mu_{\varphi} (\mathd \varphi) . \label{eq:integral}
  \end{equation}
  with
  
\begin{multline*} S (\varphi) = \int_{- \infty}^0 \left[ \frac{f (t)}{2} V''
     (\varphi (t)) - \frac{1}{2} (f (t) V' (\varphi (t)))^2 - 2 f' (t) V (\phi
     (t)) \right] \mathd t+\\
 - \int_{- \infty}^0 f (t) V' (\varphi (t)) \circ
     \mathd B (t) .\end{multline*}
  
  Here $(B (t))_{t \in \mathbb{R}}$ is the double sided Brownian motion
  (adapted with respect to $\varphi$) such that $\partial_t B = \xi =
  (\partial_t + m^2) \varphi$ and $\circ \mathd B$ denotes the corresponding
  Stratonovich integral.
  
  Parisi and Sourlas~{\cite{parisi_supersymmetric_1982}} observed long ago
  that the r.h.s. of eq.~{\eqref{eq:integral}} admits a representation using a
  Gaussian super-field $\Phi$ defined on the superspace $(t, \theta,
  \bar{\theta})$ where $t$ is the usual time variable and $\theta,
  \bar{\theta}$ are two Grassmann variables playing the role of additional
  ``fermionic'' spatial coordinates (see Section~\ref{sec:supersymm} for the
  necessary notions and notations). For the moment let us simply remark that
  $\Phi$ can be rigorously constructed as a random field on a non-commutative
  probability space with expectation denoted by $\langle \cdot \rangle$ in
  such a way that expectation of polynomials in $\Phi$ can be reduced, via an
  analog of Wick's theorem, to linear combinations of products of covariances.
  If the covariance of the super-field has the form
  \begin{equation}
    \langle \Phi (t, \theta, \bar{\theta}) \Phi (s, \theta', \bar{\theta}')
    \rangle = \frac{1}{2 m^2} \mathcal{G} (| t - s |) + \mathcal{G} (t - s)
    (\theta' - \theta) \bar{\theta}' - \mathcal{G} (s - t) (\theta' - \theta)
    \bar{\theta} \label{eq:covariance1},
  \end{equation}
  then we will prove in Theorem~\ref{theorem_limit2} below that the following
  representation formula holds
  \begin{equation}
    \int F (\varphi (0)) \exp (S (\varphi)) \mu_{\varphi} (\mathd \varphi) =
    \left\langle F (\Phi (0)) \exp \left( \int_{- \infty}^0 f (t + 2 \theta
    \bar{\theta}) V (\Phi (t, \theta, \bar{\theta})) \mathd t \mathd \theta
    \mathd \bar{\theta} \right) \right\rangle . \label{eq:integral2}
  \end{equation}
  Note that in the l.h.s. we have usual (commutative) probabilistic objects
  while the r.h.s. is expressed in the language of non-commutative
  probability.
  
  The interest of this reformulation lies in the fact that on the superspace
  $(t, \theta, \bar{\theta})$ one can define supersymmetric transformations
  which preserve the quantity $t + 2 \theta \bar{\theta}$. Integrals of
  supersymmetric quantities satisfy well known localization (also called
  dimensional reduction)
  formulas~{\cite{brydges_branched_2003,helmuth_dimensional_2016,brydges_dimensional_2003,Klein1984,Zaboronsky1997}}
  which express integrals over the superspace as evaluations in zero, more
  precisely if $F = f (t + 2 \theta \bar{\theta}) \in \mathcal{S}
  (\mathfrak{S})$ is a supersymmetric function and $T \in \mathcal{S}'
  (\mathfrak{S})$ is a supersymmetric distribution we have that
  \[ \int_{- \infty}^K T (t, \theta, \bar{\theta}) \cdot F (t, \theta,
     \bar{\theta}) \mathd t \mathd \theta \mathd \bar{\theta} = - 2
     T_{\emptyset} (K) F_{\emptyset} (K) \]
  for any $K \in \mathbb{R}$ (see Theorem~\ref{theorem_supersymmetry1} for a
  precise statement).
  
  We cannot apply Theorem~\ref{theorem_supersymmetry1} directly to
  expression~{\eqref{eq:integral2}} since the superfield $\Phi$ is not
  supersymmetric. On the other hand the correlation
  function~{\eqref{eq:covariance1}} is supersymmetric with respect to $(s,
  \theta', \bar{\theta}')$ when $t \geqslant s$ and with respect to $(t,
  \theta, \bar{\theta})$ when $t \leqslant s$. This property and the
  Markovianity of the kernel $\mathcal{G}$, namely
  \[ \mathcal{G} (t - s) \mathcal{G} (s - t) = 0 \]
  when $s \not = t$, allows us to prove a localization property for the
  \tmtextit{expectation of \ supersymmetric linear functionals of $\Phi$} (see
  Theorem~\ref{theorem_localization}), namely we prove that
  \[ \left\langle F (\Phi (0)) \exp \left( \int_{- \infty}^0 f (t + 2 \theta
     \bar{\theta}) V (\Phi (t, \theta, \bar{\theta})) \mathd t \mathd \theta
     \mathd \bar{\theta} \right) \right\rangle = \langle F (\Phi (0)) \exp [-
     2 V (\Phi (0))] \rangle . \]
  Since $\Phi (0)$ is distributed as a Gaussian with mean $0$ and variance $2
  m^{- 2}$ this implies the claim.
\end{proof*}

The rest of the paper contains details on the definition of the super-fields
and the proofs of the intermediate results.

\

\tmtextbf{Acknowledgement.} The authors are funded by the DFG under Germany's Excellence Strategy - GZ 2047/1, Project-ID 390685813.  The second author is supported by DFG via CRC 1060.

\section{Super-geometry and Gaussian super-fields}\label{sec:supersymm}

\subsection{Some notions of super-geometry}

We denote by $\mathfrak{S}$ an infinite dimensional Grassmannian algebra
generated by an enumerable number of free generators $\{ 1, \theta_1,
\theta_2, \ldots, \theta_n, \ldots \}$. With this we means that any element of
$\Theta \in \mathfrak{S}$ can be written in a unique way using a finite number
of sum and products between the generators $\theta_i$. The product between
$\theta_i$ is anti-commuting which means that $\theta_i \theta_j = - \theta_j
\theta_i$ and they commute with $1$. We call $\mathfrak{S}_0 = \tmop{span} \{
1 \}$, $\mathfrak{S}_1 = \tmop{span} \{ \theta_1, \theta_2, \ldots, \theta_n,
\ldots \}$ and with $\mathfrak{S}_k = \tmop{span} \{ \theta_{i_1} \cdots
\theta_{i_k} | \theta_{i_j} \in \mathfrak{S}_1 \}$.

If $\theta_1, \ldots, \theta_h \in \mathfrak{S}_1$ we denote by $\mathfrak{S}
(\theta_1, \ldots, \theta_h)$ the finite dimensional Grassmannian sub-algebra
of $\mathfrak{S}$ generated by $\{ 1, \theta_1, \ldots, \theta_h \}$, and we
denote by $\mathfrak{O}_h$ the universal Grassmannian algebra generated by $h$
elements. We suppose that there is an order between the generators of
$\mathfrak{O}_h$. Once we fixed an order between $\theta_1, \ldots, \theta_h$
there is a natural isomorphism between $\mathfrak{O}_h$ and $\mathfrak{S}
(\theta_1, \ldots, \theta_h)$.

We can define a notion of smooth function $F : \mathbb{R}^n \times
\mathfrak{S}_1^h \rightarrow \mathfrak{S}$. Let $\tilde{F}$ be a smooth
function from $\mathbb{R}^n$ taking values in $\mathfrak{O}_h$ which means an
object of the form
\[ \tilde{F} (x) = \widetilde{F_{}}_{\emptyset} (x) 1 + \sum_{i = 1}^h
   \tilde{F}_i (x) \mathfrak{t}_i + \sum_{1 \leqslant i < j \leqslant h}
   \tilde{F}_{i, j} (x) \mathfrak{t}_i \mathfrak{t}_j + \cdots + \tilde{F}_{1,
   2, \ldots, h} (x) \mathfrak{t}_1 \cdots \mathfrak{t}_h . \]
We define $F$ associated with $\tilde{F}$ in the following way: $F$ associates
to $(x, \theta_1, \ldots, \theta_h) \in \mathbb{R}^n \times \mathfrak{S}_1^h$
the element $\tilde{F} (x) \in \mathfrak{S} (\theta_1, \ldots, \theta_h)$
where we make the identification of $\mathfrak{S} (\theta_1, \ldots,
\theta_h)$ with $\mathfrak{O}_h$, i.e.
\[ F (x, \theta_1, \ldots, \theta_h) = \widetilde{F_{\emptyset}}_{} (x) 1 +
   \sum_{i = 1}^h \tilde{F}_i (x) \theta_i + \sum_{1 \leqslant i < j \leqslant
   h} \tilde{F}_{i, j} (x) \theta_i \theta_j + \cdots + \tilde{F}_{1, 2,
   \ldots, h} (x) \theta_1 \cdots \theta_h . \]
Hereafter we use the notation $F_{\emptyset}, F_{\theta_1}, \ldots$ for
denoting $F_{\emptyset} = \tilde{F}_{\emptyset}, F_{\theta_1} = \tilde{F}_1,
\ldots$ .We say that $F$ is a Schwartz function if $F_{\emptyset}, F_i,
\ldots$ are Schwartz functions. We denote by $\mathcal{S} (\mathfrak{S}^h)$
the set of Schwartz functions with $h$ anti-commuting variables.

If $H : \mathbb{R} \rightarrow \mathbb{R}$ is a smooth function we can define
the composition $H \circ F$ in the following way
\[ \begin{array}{ll}
     H \circ F (x, \theta_1, \ldots, \theta_h) = & H (F_{\emptyset} (x)) 1 +
     H' (F_{\emptyset} (x)) (F (x, \theta_1, \ldots, \theta_h) - F_{\emptyset}
     (x) 1)\\
     & + \frac{1}{2} H'' (F_{\emptyset} (x)) (F (x, \theta_1, \ldots,
     \theta_h) - F_{\emptyset} (x) 1)^2\\
     & \cdots + \frac{1}{h!} H^{(h)} (F_{\emptyset} (x)) (F (x, \theta_1,
     \ldots, \theta_h) - F_{\emptyset} (x) 1)^h .
   \end{array} \]
On $\mathfrak{S}$ is possible to define a notion of integral called Berezin
integral, in the following way $\int \theta \mathd \theta_{} = 1$, $\int
\bar{\theta} \mathd \theta = 0$ if $\bar{\theta} \in \mathfrak{S}_1$ and
$\bar{\theta} \not = \theta,$ $\int \Theta \theta \mathd \theta = \Theta$
where $\Theta = \theta_1 \cdots \theta_h \in \mathfrak{S}_h$ and $\theta_i
\not = \theta$ and $\int \cdot \mathd \theta$ is linear in its argument. The
integral $\int \Theta \mathd \theta_1 \mathd \theta_2 \cdots \mathd \theta_h$
is defined as $\int \left( \cdots \left( \int \left( \int \Theta \mathd
\theta_1 \right) \mathd \theta_2 \right) \cdots \right) \mathd \theta_h$.

If $F$ is a smooth function we can define the integral of $F$ with respect to
$\mathd \theta_1 \cdots d \theta_h$ in the following way $\int F (x, \theta_1,
\ldots, \theta_h) \mathd x \mathd \theta_1 \cdots \mathd \theta_h$ first
applying the integral $\int \cdot \mathd x$ to $F_{\emptyset}, F_i, \ldots$
obtaining an element of $\mathfrak{S} (\theta_1, \ldots, \theta_h)$ and then
applying the Berezin integral to this result. Using this notion of integral
and the induced duality between smooth functions, it is possible to define the
notion of tempered distribution $T \in \mathcal{S}' (\mathfrak{S}^h)$. The
distribution $T$ is an object of the form
\[ T (x, \theta_1, \ldots, \theta_k) = T_{\emptyset} (x) 1 + \sum_{i = 1}^h
   T_{\theta_i} (x) \theta_i + \cdots + T_{\theta_1 \ldots \theta_h} (x)
   \theta_1 \cdots \theta_h \]
where $T_{\emptyset} (x)$, $T_{\theta_i} (x)$, {\textdots}, $T_{\theta_1
\ldots \theta_h} (x)$ are Schwartz distributions.

\subsection{Construction of the super-field}

Following the analogous construction
in~{\cite{Klein1984,albeverio_elliptic_2018}} the super-field $\Phi$ is
defined as
\[ \Phi (t, \theta, \bar{\theta}) = \varphi (t) + \bar{\psi} (t) \theta + \psi
   (t) \bar{\theta} + \omega (t) \theta \bar{\theta}, \]
where $\varphi, \psi, \bar{\psi}, \omega$ are complex Gaussian fields realized
as functional from $\mathcal{S} (\mathbb{R})$ into the set of operators
$\mathcal{O} (\mathfrak{H})$ on a complex vector space $\mathfrak{H}$ with a
fixed state $\Omega$ (hereafter we denote by $\langle a \rangle_{\Omega} =
\langle \Omega, a (\Omega) \rangle_{\mathfrak{H}}$ for any $a \in \mathcal{O}
(\mathfrak{H})$), and $\theta, \bar{\theta}$ are any pair of anti-commuting
variables $\theta, \bar{\theta} \in \mathfrak{S}$ commuting with the operators
$\omega, \varphi$ and anti-commuting with the operators $\psi, \bar{\psi}$.

The Gaussian fields $\varphi, \psi, \bar{\psi}, \omega$ mush be realized as
operators defined from $\mathcal{S} (\mathbb{R})$ taking values in
$\mathcal{O} (\mathfrak{H})$, for a suitable Hilbert space $\mathfrak{H}$ with
a state $\Omega \in \mathfrak{H}$ such that the condition
{\eqref{eq:covariance1}} holds. Making a formal computation we obtain that
\begin{eqnarray}
  \langle \Phi (t, \theta, \bar{\theta}) \Phi (s, \theta', \bar{\theta}')
  \rangle_{\Omega} & = & \langle \varphi (t) \varphi (s) \rangle_{\Omega} -
  \langle \bar{\psi} (t) \psi (s) \rangle_{\Omega} \theta \bar{\theta}' -
  \langle \psi (t) \bar{\psi} (s) \rangle_{\Omega} \bar{\theta} \theta' +
  \nonumber\\
  &  & + \langle \varphi (t) \omega (s) \rangle_{\Omega} \theta'
  \bar{\theta}' + \langle \omega (t) \varphi (s) \rangle_{\Omega} \theta
  \bar{\theta} + \langle \omega (t) \omega (s) \rangle_{\Omega} \theta
  \bar{\theta} \theta' \bar{\theta}' \nonumber
\end{eqnarray}
from which we get
\begin{equation}
  \begin{array}{ccc}
    \langle \varphi (t) \varphi (s) \rangle_{\Omega} = \frac{1}{2 m^2}
    \mathcal{G} (| t - s |) & \langle \bar{\psi} (t) \psi (s) \rangle_{\Omega}
    = \mathcal{G} (t - s) & \langle \varphi (t) \omega (s) \rangle_{\Omega} =
    \mathcal{G} (t - s)\\
    & \langle \omega (t) \omega (s) \rangle_{\Omega} = 0. & 
  \end{array} \label{eq:covariance2}
\end{equation}
Using the commutation relations
\begin{equation}
  \begin{array}{lcl}
    \{ \varphi (t), \varphi (s) \}_+ = 0 & \{ \varphi (t), \omega (s) \}_+ = 0
    & \{ \omega (t), \omega (s) \}_+ = 0
  \end{array} \label{eq:commutator1}
\end{equation}
\begin{equation}
  \begin{array}{c}
    \{ \bar{\psi} (t), \psi (s) \}_- = \{ \psi (t), \psi (s) \}_- = \{
    \bar{\psi} (t), \bar{\psi} (s) \}_- = 0\\
    \{ \varphi (t) \psi (s) \}_+ = \{ \varphi (t) \bar{\psi} (s) \}_+ = \{
    \omega (t), \psi (s) \}_+ = \{ \omega (t), \bar{\psi} (s) \}_+ = 0
  \end{array} \label{eq:commutator2}
\end{equation}
where $\{ K_1, K_2 \}_+ = K_1 K_2 - K_2 K_1$ and $\{ K_1, K_2 \}_- = K_1 K_2 +
K_2 K_1$ (where $K_1, K_2 \in \mathcal{B} (\mathfrak{H})$) are the commutator
and the anti-commutator of closed operators having a non void common core. By
Wick theorem (see, e.g.~{\cite{Fetter2012}}~Chapter~3~Section~8) the
expectation of arbitrary polynomials in $\varphi, \psi, \bar{\psi}, \omega$ is
completely determined.

\

The bosonic field $\varphi$ is a standard (real and commutative) Gaussian
field with covariance $\mathcal{G} (| t - s |)$. Also $\omega$ is a standard
(complex and commutative) Gaussian field of the form
\[ \omega (t) = \xi (t) + i \eta (t), \]
where $\xi = (\partial_t + m^2) \varphi$ and $\eta$ is a Gaussian white noise
with Cameron-Martin space $L^2 (\mathbb{R})$ independent of $\varphi$. We can
realize the Gaussian field $\varphi, \omega$ as (unbounded) operators defined
on a Hilbert space $\mathfrak{\mathfrak{H}_{\varphi, \omega}}$ and with a
state $\Omega_{\varphi, \omega}$. We can take $\mathfrak{H}_{\varphi, \omega}
= L^2 (\mu_{\varphi, \omega})$ where $\mu_{\varphi, \omega}$ is the law of
$(\varphi, \omega)$ on $C (\mathbb{R}) \times \mathcal{S}'_{\mathbb{C}}
(\mathbb{R})$ and $\Omega_{\varphi, \omega} = 1$.

\

The fermionic fields $\psi, \bar{\psi}$ are build as follows. Let $a, b$ and
$a^{\ast}, b^{\ast}$ be two construction and annihilation operators defined as
bounded functional on $\mathcal{S} (\mathbb{R})$ taking value in $\mathcal{B}
(\mathfrak{H}_{\psi, \bar{\psi}})$ (where $\mathfrak{H}_{\psi, \bar{\psi}}$ is
a suitable Hilbert space with a fixed state $\Omega_{\psi, \bar{\psi}}$) such
that
\begin{eqnarray}
  & \{ a (f), a (g)^{} \}_- = \{ b (f), b (g)^{} \}_- = 0 &  \nonumber\\
  & \{ a (f), b (g)^{} \}_- = \{ a^{\ast} (f), b (g)^{} \}_- = 0 & 
  \nonumber\\
  & \{ a^{\ast} (g), a (f) \}_- = \{ b^{\ast} (g), b (f) \}_- = \left(
  \int_{\mathbb{R}} f (t) g (t) \mathd t \right) I_{\mathfrak{H}_{\psi,
  \bar{\psi}}}, &  \nonumber
\end{eqnarray}
for any $f, g \in \mathcal{S} (\mathbb{R})$, and such that
\[ \langle a (f) K \rangle_{\Omega_{\psi, \bar{\psi}}} = \langle K a^{\ast}
   (f) \rangle_{\Omega_{\psi, \bar{\psi}}} = \langle b (f) K
   \rangle_{\Omega_{\psi, \bar{\psi}}} = \langle K b^{\ast} (f)
   \rangle_{\Omega_{\psi, \bar{\psi}}} = 0, \]
where $K$ is any bounded operator $K \in \mathcal{B} (\mathfrak{H}_{\psi,
\bar{\psi}})$. We define $\mathcal{U} : \mathcal{S} (\mathbb{R}) \rightarrow
\mathcal{S} (\mathbb{R})$ as
\[ \mathcal{U} (f) (t) = \frac{1}{2 \pi} \int_{\mathbb{R}} \frac{e^{- i \xi
   t}}{i \xi + m^2} \hat{f} (\xi) \mathd \xi . \]
We then write
\[ \psi (f) = a^{\ast} (\mathcal{U}^{\ast} (f)) + b (f), \quad \bar{\psi} (f)
   = b^{\ast} (\mathcal{U} (f)) - a (f), \]
where $\mathcal{U}^{\ast}$ is the adjoint of $\mathcal{U}$ with respect to the
Lebesgue measure on $\mathbb{R}^2$. In this way, we have
\[ \{ \bar{\psi} (t), \psi (s) \}_- = \{ \psi (t), \psi (s) \}_- = \{
   \bar{\psi} (t), \bar{\psi} (s) \}_- = 0, \]
and also
\[ \begin{array}{rl}
     \langle \bar{\psi} (f_{}) \psi (g) \rangle_{\Omega_{\psi, \bar{\psi}}} =
     & \langle b^{\ast} (f) a^{\ast} (g) \rangle_{\Omega_{\psi, \bar{\psi}}} +
     \langle b^{\ast} (f) b (g) \rangle_{\Omega_{\psi, \bar{\psi}}} - \langle
     a (f) a^{\ast} (g) \rangle_{\Omega_{\psi, \bar{\psi}}}+\\
     & - \langle a (f) b (g) \rangle_{\Omega_{\psi, \bar{\psi}}} =
     \int_{\mathbb{R}} \mathcal{U} (f) (t) g (t) \mathd t\\
     = & \int_{\mathbb{R}} g (t) \int_{- \infty}^t e^{- m^2 (t - s)} f (s)
     \mathd s \mathd t = \int_{\mathbb{R}^2} g (t) \mathcal{G} (t - s) f (s)
     \mathd s \mathd t.
   \end{array} \]
In other words we have $\langle \bar{\psi} (t) \psi (s) \rangle_{\Omega_{\psi,
\bar{\psi}}} = \mathcal{G} (t - s)$ as required. We can define the operators
$\varphi, \psi, \bar{\psi}, \omega$ on a unique (quantum) probability space
taking
\[ \mathfrak{H}=\mathfrak{H}_{\varphi, \omega} \otimes \mathfrak{H}_{\psi
   \comma \bar{\psi}} \qquad \Omega = \Omega_{\varphi, \omega} \otimes
   \Omega_{\psi, \bar{\psi}} . \]
In order to realize the field $\Phi$ in a rigorous way we consider a complex
sub-algebra $\mathfrak{A} \subset \mathcal{O} (\mathfrak{H})$ such that
$\varphi, \psi, \bar{\psi}, \omega$ takes values in $\mathfrak{A}$ and for any
smooth function $V : \mathbb{R} \rightarrow \mathbb{R}$ we have $V (\varphi
(g)) \in \mathfrak{A}$, where $g$ is any function in $\mathcal{S}
(\mathbb{R})$. This sub-algebra is generated (from an algebraic point of view)
by operators of the form $V (\varphi (g))$, $\omega (g)$, $\psi (g)$,
$\bar{\psi} (g)$ and $I_{\mathfrak{H}}$. We consider the vector space
$\mathcal{A} = \mathfrak{A} \times \mathfrak{S}$. There are two preferred
hyperplane $\mathcal{A}_{\mathfrak{A}}$ and $\mathcal{A}_{\mathfrak{S}}$
defined as
\[ \mathcal{A}_{\mathfrak{A}} = \{ (a, 1_{\mathfrak{S}}) |a \in \mathfrak{A}
   \}, \quad \mathcal{A}_{\mathfrak{S}} = \{ (I_{\mathfrak{H}}, \theta) |
   \theta \in \mathfrak{S} \}, \]
with the natural immersions $i_{\mathfrak{A}} : \mathfrak{A} \rightarrow
\mathcal{A}$ and $i_{\mathfrak{S}} : \mathfrak{S \rightarrow} \mathcal{A}$
defined as $i_{\mathfrak{A}} (a) = (a, 1_{\mathfrak{S}})$ and
$i_{\mathfrak{S}} (\theta) = (I_{\mathfrak{H}}, \theta)$ (we note that
$\mathcal{A}_{\mathfrak{A}} \assign i_{\mathfrak{A}} (\mathfrak{A})$ and
$\mathcal{A}_{\mathfrak{S}} \assign i_{\mathfrak{S}} (\mathfrak{S})$). It is
clear that $\mathcal{A}_{\mathfrak{A}}, \mathcal{A}_{\mathfrak{S}}$ generates
the whole $\mathcal{A}$. On $\mathcal{A}$ we define the following product
$\cdot$, in such a way that the maps $i_{\mathfrak{A}}$ and $i_{\mathfrak{S}}$
respect the product (i.e. $i_{\mathfrak{A}} (a b) = i_{\mathfrak{A}} (a) \cdot
i_{\mathfrak{A}} (b)$ and $i_{\mathfrak{S}} (\theta_1 \theta_2) =
i_{\mathfrak{S}} (\theta_1) \cdot i_{\mathfrak{S}} (\theta_2)$) and such that
\[ (V (\varphi (g)), 1_{\mathfrak{S}}) \cdot (I_{\mathfrak{H}}, \theta) =
   (I_{\mathfrak{H}}, \theta) \cdot (V (\varphi (g)), 1_{\mathfrak{S}}) = (V
   (\varphi (g)), \theta) \]
\[ (\omega (g), 1_{\mathfrak{S}}) \cdot (I_{\mathfrak{H}}, \theta) =
   (I_{\mathfrak{H}}, \theta) \cdot (\omega (g), 1_{\mathfrak{S}}) = (\omega
   (g), \theta) \]
\[ (\psi (g), 1_{\mathfrak{S}}) \cdot (I_{\mathfrak{H}}, \theta) = -
   (I_{\mathfrak{H}}, \theta) \cdot (\psi (g), 1_{\mathfrak{S}}) = (\psi (g),
   \theta) \]
\[ (\bar{\psi} (g), 1_{\mathfrak{S}}) \cdot (I_{\mathfrak{H}}, \theta) = -
   (I_{\mathfrak{H}}, \theta) \cdot (\bar{\psi} (g), 1_{\mathfrak{S}}) =
   (\bar{\psi} (g), \theta) \]
where $g \in \mathcal{S} (\mathbb{R})$ and $\theta \in \mathfrak{S}_1$ (not in
$\mathfrak{S}$). The product $\cdot$ can be uniquely extended (in a
associative way) on $\mathcal{A}$ since $\mathcal{A}_{\mathfrak{A}},
\mathcal{A}_{\mathfrak{S}}$ generates the whole $\mathcal{A}$, operators of
the form $V (\varphi (g))$, $\omega (g)$, $\psi (g)$, $\bar{\psi} (g)$
generates the whole $\mathfrak{A}$ and $\mathfrak{S}_1$ generates the whole
$\mathfrak{S}$. Hereafter we will omit to explicitly write the product $\cdot$
if this omission does not cause any confusion. \

\

On $\mathcal{A}$ we can define a linear operator $\langle \cdot \rangle :
\mathcal{A} \rightarrow \mathcal{A}_{\mathfrak{S}} \backsimeq \mathfrak{S}$
defined as
\[ \langle (a, \theta) \rangle = \langle a \rangle_{\Omega}
   (I_{\mathfrak{H}}, \theta) . \]
Furthermore for any $\theta_1, \ldots, \theta_n \in \mathcal{} \mathfrak{S}_1$
we define the linear operator $\int \cdot \mathd \theta_1 \ldots \mathd
\theta_n : \mathcal{A} \rightarrow \mathcal{A}_{\mathfrak{A}}$ such$\int \cdot
\mathd \theta_1 \ldots \mathd \theta_n |_{\mathcal{A}_{\mathfrak{S}}}$ is the
usual Berezin integral induced by the identification
$\mathcal{A}_{\mathfrak{S}} \backsimeq \mathfrak{S}$
\[ \int (a, \theta) \mathd \theta_1 \ldots \mathd \theta_n = \left( \int
   \theta \mathd \theta_1 \ldots \mathd \theta_n \right) (a, 1_{\mathfrak{S}})
   . \]
Hereafter we identify the space $\mathfrak{A}$ and $\mathfrak{S}$ with
$\mathcal{A}_{\mathfrak{A}}$ and $\mathcal{A}_{\mathfrak{S}}$ respectively,
and we write instead of $(a, 1_{\mathfrak{S}})$, $(I_{\mathfrak{H}}, \theta)$,
$(I_{\mathfrak{H}}, 1_{\mathfrak{S}})$ simply $a$, $\theta$ and $1$
respectively (in this way we take also the tacit identification of
$\tmop{span} \{ 1_{\mathfrak{S}} \} = \mathfrak{S}_0$ with $\mathbb{R}$).
Furthermore we identify $\varphi, \psi, \bar{\psi}, \omega$ with
$i_{\mathfrak{A}} \circ \varphi, i_{\mathfrak{A}} \circ \psi, i_{\mathfrak{A}}
\circ \bar{\psi}, i_{\mathfrak{A}} \circ \omega$.

\begin{remark}
  Since $\psi, \bar{\psi}$ are ``independent'' with respect to $\varphi$ and
  $\omega$ (since their can be realized on a space of the form
  $\mathfrak{H}=\mathfrak{H}_{\varphi, \omega} \otimes \mathfrak{H}_{\psi
  \comma \bar{\psi}}$) it is well defined the expectation with respect to the
  fields $\psi, \bar{\psi}$ only, namely is well defined an operator $\langle
  \cdot \rangle_{\psi, \bar{\psi}} : \mathcal{O_{\mathfrak{H}}} \rightarrow
  \mathcal{O}_{\mathfrak{\varphi, \omega}}$ such that
  \[ \begin{array}{l}
       \langle V (\varphi (t_1), \ldots, \varphi (t_k)) \psi (t_1') \bar{\psi}
       (t''_1) \cdots \psi (t_{k'}') \bar{\psi} (t''_{k'}) \rangle_{\psi,
       \bar{\psi}}\\
       \qquad \qquad = V (\varphi (t_1), \ldots, \varphi (t_k)) \langle \psi
       (t_1') \bar{\psi} (t''_1) \cdots \psi (t_{k'}') \bar{\psi} (t''_{k'})
       \rangle . \nobracket
     \end{array} \]
  We can extend the operator $\langle \cdot \rangle_{\psi, \bar{\psi}}$ to
  $\mathcal{A}$ in the way the operator $\langle \cdot \rangle$ is defined on
  $\mathcal{A}$.
\end{remark}

\subsection{Relation with SDEs}

In this section we want to use the super-field $\Phi$ for representing the
solution to the SDE~{\eqref{eq:SDEmain}} through the
integral~{\eqref{eq:integral}}.

\

First of all we have to define the notion of composition of the super-field
$\Phi$ with smooth functions. Consider the smooth function $H : \mathbb{R}
\rightarrow \mathbb{R}$ growing at most exponentially at infinity. We can
formally expand $H$ in Taylor series and using the properties of $\theta,
\bar{\theta}$ we obtain
\[ \begin{array}{rcl}
     H (\Phi (t, \theta, \bar{\theta})) & = & H (\varphi (t)) + H' (\varphi
     (t)) \bar{\psi} (t) \theta + H' (\varphi (t)) \psi (t) \bar{\theta} +\\
     &  & \quad + (H' (\varphi (t)) \omega (t) + H'' (\varphi (t)) \psi (t)
     \bar{\psi} (t)) \theta \bar{\theta} .
   \end{array} \]
Unfortunately the products $H' (\varphi (t)) \omega (t)$ and $H'' (\varphi
(t)) \psi (t) \bar{\psi} (t)$ are ill defined since the factors are not
regular enough. For this reason we consider a symmetric mollifier $\rho :
\mathbb{R} \rightarrow \mathbb{R}_+$ (with $\rho (t) = \rho (- t)$) and the
field $\Phi_{\epsilon} = \rho_{\epsilon} \ast \Phi$, where $\rho_{\epsilon}
(t) = \epsilon^{- 1} \rho (t \epsilon^{- 1})$. If $G$ is a super-function, $F$
is a smooth function and $\mathbb{K}$ is an entire function we define
\begin{equation}
  \begin{array}{l}
    \left\langle F (\varphi (0)) \mathbb{K} \left( \int G (t, \theta,
    \bar{\theta}) H (\Phi (t, \theta, \bar{\theta})) \mathd t \mathd \theta
    \mathd \bar{\theta} \right) \right\rangle\\
    \qquad \qquad \qquad \assign \lim_{\epsilon \rightarrow 0} \left\langle F
    (\varphi_{\varepsilon} (0)) \mathbb{K} \left( \int G (t, \theta,
    \bar{\theta}) H (\Phi_{\epsilon} (t, \theta, \bar{\theta})) \mathd t
    \mathd \theta \mathd \bar{\theta} \right) \right\rangle .
  \end{array} \label{eq:epsilon}
\end{equation}
We want to prove that the previous expression is well defined and does not
depend on $\rho$.

\begin{remark}
  It is important to note that the expression~{\eqref{eq:epsilon}} does not
  depend on $\rho$ only if $\rho$ is symmetric. If we choose a different
  $\rho$ (such that for example $\int_{- \infty}^0 \rho \mathd t \not =
  \int_0^{+ \infty} \rho \mathd t$) we will obtain a different limit. This is
  due to the fact that the products $H' (\varphi (t)) \omega (t)$ and $H''
  (\varphi (t)) \psi (t) \bar{\psi} (t)$ are ill defined and it is analogous
  to the possibility to obtain Ito or Stratonovich integral in stochastic
  calculus considering different approximations of the stochastic integral.
\end{remark}

\begin{lemma}
  \label{lemma_limit1}Let $F_1, \ldots, F_n : \mathbb{R} \times \mathbb{R}
  \rightarrow \mathbb{R}$ be smooth functions with compact support in the
  first variable and growing at most exponentially at infinity in the second
  variable then we have
  \begin{equation}
    \lim_{\epsilon \rightarrow 0} \left\langle \prod_{i = 1}^n \int F_i (t,
    \varphi_{\epsilon} (t)) \bar{\psi}_{\epsilon} (t) \psi_{\epsilon} (t)
    \mathd t \right\rangle_{\psi, \bar{\psi}} = \int \prod_{i = 1}^n F_i (t_i,
    \varphi (t)) \mathfrak{G}_n (t_1, \ldots, t_n) \mathd t_1 \ldots \mathd
    t_n . \label{eq:determinant}
  \end{equation}
  in $L^p (\mu_{\varphi})$. Here $\mathfrak{G}_n (t_1, \ldots, t_n) = \det
  ((G_{i, j})_{i, j = 1, \ldots, n})$ with $G_{i, j} = \mathcal{G} (t_j -
  t_i)$ if $i \neq j$ and $G_{i, i} = 1 / 2$.
\end{lemma}

\begin{proof}
  It is simple to see that $\lim_{\epsilon \rightarrow 0} \langle
  \bar{\psi}_{\epsilon} (t_1) \psi_{\epsilon} (t_2) \rangle_{\psi, \bar{\psi}}
  = \mathcal{G} (t_1 - t_2)$ when $t_1 \not = t_2$ and \ $\lim_{\epsilon
  \rightarrow 0} \langle \bar{\psi}_{\epsilon} (t) \psi_{\epsilon} (t)
  \rangle_{\psi, \bar{\psi}} = \frac{1}{2}$ (this is due to the fact that
  $\rho (t) = \rho (- t)$). Since $\langle \bar{\psi}_{\epsilon} (t_1)
  \psi_{\epsilon} (t_1) \cdots \bar{\psi}_{\epsilon} (t_2) \psi_{\epsilon}
  (t_n) \rangle_{\psi, \bar{\psi}}$ is uniformly bounded in $t$ and $\epsilon$
  and $F_i (t, \varphi_{\epsilon} (t))$ is uniformly bounded in $L^p
  (\mu_{\varphi})$ in $t$ and $\epsilon$ the claim follows.
\end{proof}

\begin{remark}
  \label{remark_limit}Since only one between $\mathcal{G} (t - s)$ and
  $\mathcal{G} (s - t)$ is non zero if $F_1 = F_2 = \cdots = F_n$ then
  \[ \lim_{\epsilon \rightarrow 0} \left\langle \left( \int F_1 (t,
     \varphi_{\epsilon} (t)) \bar{\psi}_{\epsilon} (t) \psi_{\epsilon} (t)
     \mathd t \right)^n \right\rangle_{\psi, \bar{\psi}} = \left( \frac{1}{2}
     \int F_1 (t, \varphi (t)) \mathd t \right)^n . \]
\end{remark}

\begin{lemma}
  \label{lemma_limit2}Let $F : \mathbb{R} \times \mathbb{R} \rightarrow
  \mathbb{R}$ be smooth functions with compact support in the first variable
  and growing at most exponentially at infinity in the second variable then we
  have in $L^p (\mu_{\varphi, \omega}) .$
  \[ \lim_{\epsilon \rightarrow 0} \int F (t, \varphi_{\epsilon} (t))
     \omega_{\epsilon} (t) = \int F (t, \varphi (t)) \circ \mathd B (t) + i
     \int F (t, \varphi (t)) \mathd W (t) \]
  where the first one is Stratonovich integral and the second one is Ito
  integral with respect to (double sided) Brownian motions $(B (t), W (t))_{t
  \in \mathbb{R}}$ such that $\partial_t B (t) = \xi (t)$ and \ $\partial_t W
  (t) = \eta (t)$ with $B_0 = W_0 = 0$.
\end{lemma}

\begin{proof}
  This is the Wong--Zakai
  theorem~{\cite{wong_convergence_1965,wong_relation_1965,stroock_support_1972}}.
\end{proof}

\begin{theorem}
  \label{theorem_limit1}When $\mathbb{K}$ is a polynomial and $H$ grows at
  most exponentially at infinity, or $\mathbb{K}$ is entire and $H$ is bounded
  with first and second derivative bounded, the limit~{\eqref{eq:epsilon}} is
  well defined and does not depend on the symmetric mollifier $\rho$.
\end{theorem}

\begin{proof}
  When $\mathbb{K}$ is a polynomial the thesis follows directly from
  Lemma~\ref{lemma_limit1} and Lemma~\ref{lemma_limit2}. If $\mathbb{K}$ is an
  entire function and $H$ is a bounded function with first and second
  derivatives bounded it is possible to exchange the limit in $\epsilon$ with
  the power series, since $$\left\langle \left| \left\langle \left( \int G (t,
  \theta, \bar{\theta}) H (\Phi_{\epsilon} (t, \theta, \bar{\theta}))
  \right)^k \mathd t \mathd \theta \mathd \bar{\theta} \right\rangle_{\psi,
  \bar{\psi}} \right|^p \right\rangle$$ is uniformly bounded in $\epsilon$ for
  any $p > 1$. \ 
\end{proof}

\begin{theorem}
  \label{theorem_limit2}Suppose that $G (t, \theta, \bar{\theta}) =
  G_{\emptyset} (t) + G_{\theta \bar{\theta}} (t) \theta \bar{\theta}$ and
  that $H$ is bounded with the first and second derivatives bounded then
  
  \begin{multline}
    \left\langle F (\varphi (0)) \exp \left( \int G (t, \theta, \bar{\theta})
    H (\Phi (t, \theta, \bar{\theta})) \mathd t \mathd \theta \mathd
    \bar{\theta} \right) \right\rangle =\\
    = \int F (\varphi (0)) \exp \left( \frac{1}{2} \int G_{\emptyset} (t) H''
    (\varphi (t)) \mathd t - \int G_{\emptyset} (t) H' (\varphi (t)) \circ
    \mathd \xi (t) + \right.\\
    \left. - \frac{1}{2} \int (G_{\emptyset} (t) H' (\varphi (t)))^2 \mathd t
    - \int G_{\theta \bar{\theta}} (t) H (\varphi (t)) \mathd t \right)
    \mu_{\varphi} (\mathd \varphi) .
  \end{multline}
\end{theorem}

\begin{proof}
  The proof follows from Theorem~\ref{theorem_limit1}, the multiplicative
  property of exponential, Remark~\ref{remark_limit}, and the fact that the
  Fourier transform of a process integrated with respect to an independent
  white noise can be computed explicitly and in this case gives the factor
  $\exp (- \frac{1}{2} \int (G_{\emptyset} (t) H' (\varphi (t)))^2 \mathd t)$.
\end{proof}

\section{Supersymmetry and the supersymmetric field}

\subsection{The supersymmetry}

On $C^{\infty} (\mathbb{R} \times \mathfrak{S}^2_1)$ one can introduce the
(graded) derivations
\[ Q \assign 2 \theta \partial_t + \partial_{\bar{\theta}}, \qquad \bar{Q}
   \assign 2 \bar{\theta} \partial_t - \partial_{\theta}, \]
which are such that
\[ Q (t + 2 \theta \bar{\theta}) = 0 = \bar{Q} (t + 2 \theta \bar{\theta}), \]
namely they annihilate the function $t + 2 \theta \bar{\theta}$ defined on
$\mathbb{R} \times \mathfrak{S}^2_1$. Moreover if $Q F = \bar{Q} F = 0$, for
$F$ in $C^{\infty} (\mathbb{R} \times \mathfrak{S}^2_1)$, then we must have
\[ 0 = Q F (x, \theta, \bar{\theta}) = 2 \partial_t f_{\emptyset} (t) \theta +
   f_{\bar{\theta}} (t) + \partial_t f_{\bar{\theta}} (x) \theta \bar{\theta}
   - f_{\theta \bar{\theta}} (t) \theta \]
\[ 0 = \bar{Q} F (x, \theta, \bar{\theta}) = 2 \partial_t f_{\emptyset} (t)
   \bar{\theta} + f_{\theta} (t) - \partial_t f_{\theta} (x) \theta
   \bar{\theta} - f_{\theta \bar{\theta}} (t) \bar{\theta} \]
and therefore
\[ \partial_t f_{\emptyset} (t) = \frac{1}{2} f_{\theta \bar{\theta}} (t)
   \qquad \text{and} \qquad f_{\theta} (t) = f_{\bar{\theta}} (t) = 0. \]
This means that there exists an $f \in C^{\infty} (\mathbb{R}, \mathbb{R})$
such that
\[ f (t + 2 \theta \bar{\theta}) = f (t) + 2 f' (t) \theta \bar{\theta} =
   f_{\emptyset} (t) + f_{\theta \bar{\theta}} (t) \theta \bar{\theta} = F (t,
   \theta, \bar{\theta}) . \]
Namely any function satisfying these two equations can be written in the form
\[ F (t, \theta, \bar{\theta}) = f (t + 2 \theta \bar{\theta}) . \]
Suppose that $t > 0$, if we introduce the linear transformations
\[ \tau (b, \bar{b}) \left(\begin{array}{c}
     t\\
     \theta\\
     \bar{\theta}
   \end{array}\right) = \left(\begin{array}{c}
     t + 2 \bar{b} \theta \rho + 2 b \bar{\theta} \rho\\
     \theta - b \rho\\
     \bar{\theta} + \bar{b} \rho
   \end{array}\right) \in \mathfrak{S} (\theta, \bar{\theta}, \rho) \]
for $b, \bar{b} \in \mathbb{R}$ and where $\rho \in \mathfrak{S}_1$ is a new
odd variable different from $\theta, \bar{\theta}$, then we have
\[ \left. \frac{\mathd}{\mathd a} \right|_{a = 0} \tau (a b, a \bar{b}) F (t,
   \theta, \bar{\theta}) = \left. \frac{\mathd}{\mathd a} \right|_{a = 0} F
   (\tau (a b, a \bar{b}) (t, \theta, \bar{\theta})) = (b \cdot \bar{Q} +
   \bar{b} \cdot Q) F (t, \theta, \bar{\theta}) \]
so $\tau (b, \bar{b}) = \exp (b \cdot \bar{Q} + \bar{b} \cdot Q)$ and $\tau (a
b, a \bar{b}) \tau (c b, c \bar{b}) = \tau ((a + c) b, (a + c) \bar{b})$.

In particular $F \in C^{\infty} (\mathbb{R} \times \mathfrak{S}^2)$ is
supersymmetric if and only if for any $b, \bar{b} \in \mathbb{R}$ we have
$\tau (b, \bar{b}) F = F$.

\

By duality the operators $Q, \bar{Q}$ and $\tau (b, \bar{b})$ also act on the
space $\mathcal{S}' (\mathfrak{S})$ and we say that the distribution $T \in
\mathcal{S}' (\mathfrak{S})$ is supersymmetric if it is invariant with respect
to rotations in space and $Q T = \bar{Q} T = 0$. For supersymmetric functions
and distribution the following fundamental theorem holds.

\begin{theorem}
  \label{theorem_supersymmetry1}Let $F \in \mathcal{S} (\mathfrak{S})$ and $T
  \in \mathcal{S}' (\mathfrak{S})$ such that $T_0$ is a continuous function.
  If both $F$ and $T$ are supersymmetric. Then for any $K \in \mathbb{R}$ we
  have the reduction formula
  \begin{equation}
    \int_{- \infty}^K T (t, \theta, \bar{\theta}) \cdot F (t, \theta,
    \bar{\theta}) \mathd t \mathd \theta \mathd \bar{\theta} = - 2
    T_{\emptyset} (K) F_{\emptyset} (K) . \label{eq:key}
  \end{equation}
\end{theorem}

\begin{proof}
  The proof can be found in~{\cite{Klein1984}}, Lemma~4.5 for $\mathbb{R}^2$
  and in~{\cite{Zaboronsky1997}} for the case of a general super-manifold.
  Here we give the proof only for the case where $T$ is a super-function. In
  this case we have that $T (t, \theta, \bar{\theta}) = T_{\emptyset} (t) + 2
  T_{\emptyset}' (t) \theta \bar{\theta}$ and $F (t, \theta, \bar{\theta}) =
  F_{\emptyset} (t) + 2 F_{\emptyset}' (t) \theta \bar{\theta}$ from which we
  have
  \[ \begin{array}{lll}
       T (t, \theta, \bar{\theta}) \cdot F (t, \theta, \bar{\theta}) & = &
       T_{\emptyset} (t) F_{\emptyset} (t) + 2 (T'_{\emptyset} (t)
       F_{\emptyset} (t) + T_{\emptyset} (t) F'_{\emptyset} (t)) \theta
       \bar{\theta}\\
       & = & T_{\emptyset} (t) F_{\emptyset} (t) + 2 \partial_t
       (T_{\emptyset} F_{\emptyset}) (t) \theta \bar{\theta} .
     \end{array} \]
  By definition of Berezin integral we have
  \begin{eqnarray}
    \int_{- \infty}^K T (t, \theta, \bar{\theta}) \cdot F (t, \theta,
    \bar{\theta}) \mathd x \mathd \theta \mathd \bar{\theta} & = & - 2 \int_{-
    \infty}^K \partial_t (T_{\emptyset} F_{\emptyset}) (t) \mathd t
    \nonumber\\
    & = & - 2 T_{\emptyset} (K) F_{\emptyset} (K) . \nonumber
  \end{eqnarray}
\end{proof}

\begin{remark}
  \label{remark_supersymmetry}In Theorem~\ref{theorem_supersymmetry1} we can
  assume that $F = F_{\emptyset} (t) + F_{\theta \bar{\theta}} (t) \theta
  \bar{\theta}$ and $T (t, \theta, \bar{\theta}) = T_{\emptyset} (t) +
  T_{\theta \bar{\theta}} (t) \theta \bar{\theta}$ where $F_{\theta
  \bar{\theta}} (t) = 2 F_{\emptyset}' (t)$ and $T_{\theta \bar{\theta}} (t) =
  2 T_{\emptyset}' (t)$ only for $t \leqslant K$. In some way we can consider
  supersymmetric functions only on the set $(- \infty, K]$. 
\end{remark}

\subsection{Localization of supersymmetric averages}

\begin{remark}
  \label{remark_supersymmetry2}We note that the correlation function
  \[ C^{\Phi} (t, s, \theta, \bar{\theta}) = \langle \varphi (t) \Phi (s,
     \theta, \bar{\theta}) \rangle = \frac{1}{2 m^2} \mathcal{G} (t - s) +
     \mathcal{G} (t - s) \theta \bar{\theta} \]
  is a supersymmetric function when $t \geqslant s$.
\end{remark}

\begin{lemma}
  \label{lemma_superprobability}Let $g (t)$ be smooth function with compact
  support, let $P$ be a polynomial and let $t_1 > t_2 > \cdots > t_k$ and $M =
  (m_1, \ldots, m_k) \in \mathbb{N}^k$ then
 \begin{multline*} \mathcal{H}_{\ell, P}^{M, G} (t_1, \ldots, t_k) =
\\ =\left\langle \prod_{j
     = 1}^k \varphi (t_j)^{m_j} \int_{- \infty}^{t_k} \int_{- \infty}^{\tau_1}
     \cdots \int_{- \infty}^{\tau_{\ell}} \prod_{i = 1}^{\ell} g \left( \tau_i
     + 2 \theta_i \bar{\theta}_i \right) P (\Phi (\tau_i, \theta_i,
     \bar{\theta}_i)) \mathd \tau_i \mathd \theta_i \mathd \bar{\theta}_i
     \right\rangle = \\
   = \frac{(- 2 g (t_k))^{\ell}}{\ell !} \left\langle \prod_{j = 1}^k
     \varphi (t_j)^{m_j} P (\varphi (t_k))^{\ell} \right\rangle . 
\end{multline*}
\end{lemma}

\begin{proof}
  We prove the lemma on induction on $\ell$ and for simplicity we assume that
  $P (x) = x^n$, being the general case is a straightforward generalization.
  Since the proof is essentially of combinatorial nature in the following we
  consider some ill defined objects like the products $\varphi (t) \omega (s)$
  or $\psi (t) \bar{\psi} (s)$. This fact does not change the main idea of
  proof since all the expectations with respect to the previous products are
  defined using the symmetric regularization proposed in
  Lemma~\ref{lemma_limit1} and Lemma~\ref{lemma_limit2}, i.e. all the
  following computations can be made rigorous replacing $\varphi, \omega,
  \psi$ and $\bar{\psi}$ by the regularized Gaussian fields
  $\varphi_{\epsilon}, \omega_{\epsilon}, \psi_{\epsilon}$ and
  $\bar{\psi}_{\epsilon}$ (as defined in Lemma~\ref{lemma_limit1} and
  Lemma~\ref{lemma_limit2}) and then taking the limit as $\epsilon \rightarrow
  0$. The main difference between the proof below and the one involving the
  regularized fields is that in the regularized case we have also to consider
  the contractions of the form $\omega_{\varepsilon} (t) \varphi_{\varepsilon}
  (s)$ and $\psi_{\varepsilon} (t) \bar{\psi}_{\varepsilon} (s)$ when $s < t$
  and $| s - t | < \varepsilon$. Since the contribution of this kind of term
  is proportional to the support of the mollifier $\rho_{\varepsilon}$, they
  go to zero as $\varepsilon \rightarrow 0$. Let
  \[ Y^M (t_1, \ldots, t_k) \assign \prod_{j = 1}^k \varphi (t_j)^{m_j} \]
  We have
  \begin{multline*} \mathcal{H}^{M, G}_{1, x^n} (t_1, \ldots, t_k) = \left\langle Y^M (t_1,
     \ldots, t_k) \int_{- \infty}^{t_k} g (\tau + 2 \theta \bar{\theta}) (\Phi
     (\tau, \theta, \bar{\theta}))^n \mathd \tau \mathd \theta \mathd
     \bar{\theta} \right\rangle = \\
   = \int_{- \infty}^{t_k} g (\tau + 2 \theta \bar{\theta}) \langle Y^M
     (t_1, \ldots, t_k) (\Phi (\tau, \theta, \bar{\theta}))^n \rangle \mathd
     \tau \mathd \theta \mathd \bar{\theta} . \end{multline*}
  Since $\Phi$ and $\varphi$ are Gaussian fields, by Wick theorem and by
  Remark~\ref{remark_supersymmetry2}, we have that $\langle Y^M (t_1, \ldots,
  t_k) (\Phi (\tau, \theta, \bar{\theta}))^n \rangle$ is supersymmetric in
  $(\tau, \theta, \bar{\theta})$ when $\tau \leqslant t_k$. Moreover, given
  that $G = g (t + 2 \theta \bar{\theta})$ is a supersymmetric function by
  Remark~\ref{remark_supersymmetry}, we have the thesis.
  
  Suppose now that the lemma holds for $\ell - 1 \in \mathbb{N}$, letting
  \[ H (\tau_1) \assign \int_{- \infty}^{\tau_1} \cdots \int_{-
     \infty}^{\tau_{\ell}} \prod_{i = 2}^{\ell} g (\tau_i + 2 \theta_i
     \bar{\theta}_i) (\Phi (\tau_i, \theta_i, \bar{\theta}_i))^n \mathd \tau_i
     \mathd \theta_i \mathd \bar{\theta}_i, \]
  we have
  \begin{align*}
\mathcal{H}^{M, g}_{\ell, x^n}& (t_1, \ldots, t_k) =\\
=&\left\langle Y^M
     (t_1, \ldots, t_k)  \int_{- \infty}^{t_k} \int_{- \infty}^{\tau_1} \cdots
     \int_{- \infty}^{\tau_{\ell}} \prod_{i = 1}^{\ell} g \left( \tau_i + 2
     \theta_i \bar{\theta}_i \right) (\Phi (\tau_i, \theta_i,
     \bar{\theta}_i))^n \mathd \tau_i \mathd \theta_i \mathd \bar{\theta}_i
     \right\rangle \\
 = &\int_{- \infty}^{t_k} g (\tau_1 + 2 \theta_1 \bar{\theta}_1) \langle
     Y^M (t_1, \ldots, t_k) \Phi (\tau_1, \theta_1, \bar{\theta}_1)^n H
     (\tau_1) \rangle \mathd \tau_1 \mathd \theta_1 \mathd \bar{\theta}_1 \\
  = &\int_{- \infty}^{t_k} g' (\tau_1) \mathcal{H}^{(M, n), g}_{\ell - 1,
     x^n} (t_1, \ldots, t_k, \tau_1) \mathd \tau_1 -n \int_{- \infty}^{t_k}
     \langle Y^M (t_1, \ldots, t_k) \varphi (\tau_1)^{n - 1} \omega (\tau_1) H
     (\tau_1) \rangle \cdot \\
   & \cdot g (\tau_1) \mathd \tau_1- n
     (n - 1) \int_{- \infty}^{t_k} \langle Y^M (t_1, \ldots, t_k) \varphi
     (\tau_1)^{n - 2} \psi (\tau_1) \bar{\psi} (\tau_1) H (\tau_1) \rangle g (\tau_1) \mathd \tau_1 .
\end{align*}
  Where $(M, n) = (m_1, \ldots, m_k, n)$. By induction hypothesis the first
  term in the sum is exactly
  \[ \int_{- \infty}^{t_k} g' (\tau_1) \mathcal{H}^{(M, n), g}_{\ell - 1, x^n}
     (t_1, \ldots, t_k, \tau_1) \mathd \tau_1 = \int_{- \infty}^{t_k} g'
     (\tau_1) \frac{(2 g (\tau_1))^{\ell - 1}}{(\ell - 1) !} \langle \varphi
     (\tau_1)^{\ell n} Y^M (t_1, \ldots, t_k) \rangle \mathd \tau_1 . \]
  For the second term we note that
  \begin{multline*} \left\langle \varphi (\tau_1)^{n - 1} \omega (\tau_1) Y^M (t_1, \ldots,
     t_k) \int_{- \infty}^{\tau_1} \cdots \int_{- \infty}^{\tau_{\ell}}
     \prod_{i = 2}^{\ell} g (\tau_i + 2 \theta_i \bar{\theta}_i) (\Phi
     (\tau_i, \theta_i, \bar{\theta}_i))^n \mathd \tau_i \mathd \theta_i
     \mathd \bar{\theta}_i \right\rangle = \\
   = \sum_{j = 1}^k m_j \langle \omega (\tau_1) \varphi (t_j) \rangle
     \mathcal{H}^{(M - 1_j, n - 1), g}_{\ell - 1, x^n} (t_1, \ldots, t_k,
     \tau_1) + (n - 1) \langle \varphi (\tau_1) \omega (\tau_1) \rangle
     \mathcal{H}^{(M, n - 2), g}_{\ell - 1, x^n} (t_1, \ldots, t_k, \tau_1) \end{multline*}
  where $1_j = (0, \ldots, 1, 0, \ldots, 0) \in \mathbb{N}^k$ with $1$ in the
  $j$-th position and where we used Wick theorem and the fact that
  \[ \text{$\langle \varphi (\tau_1) \omega (\tau_1) \rangle =
     \frac{1}{2}$\quad and $\left\langle \omega (\tau_1) \int_{-
     \infty}^{\tau_1} \cdots \int_{- \infty}^{\tau_{\ell}} \prod_{i =
     2}^{\ell} g (\tau_i + 2 \theta_i \bar{\theta}_i) (\Phi (\tau_i, \theta_i,
     \bar{\theta}_i))^n \mathd \tau_i \mathd \theta_i \mathd \bar{\theta}_i
     \right\rangle = 0$} . \]
  Furthermore for the third we have
  \[ \left\langle \varphi (\tau_1)^{n - 2} \psi (\tau_1) \bar{\psi} (\tau_1)
     \prod_{j = 1}^k \varphi (t_j)^{m_j} \int_{- \infty}^{\tau_1} \cdots
     \int_{- \infty}^{\tau_{\ell}} \prod_{i = 2}^{\ell} g (\tau_i + 2 \theta_i
     \bar{\theta}_i) (\Phi (\tau_i, \theta_i, \bar{\theta}_i))^n \mathd \tau_i
     \mathd \theta_i \mathd \bar{\theta}_i \right\rangle = \]
  \[ = \langle \psi (\tau_1) \bar{\psi} (\tau_1) \rangle \mathcal{H}^{(M, n -
     2), g}_{\ell - 1, x^n} (t_1, \ldots, t_k, \tau_1) . \]
  In this way we obtain that
  \[ \mathcal{H}^{M, g}_{\ell, x^n} (t_1, \ldots, t_k) = (- 1)^{\ell - 1}
     2^{\ell - 1} \int_{- \infty}^{t_k} g' (\tau_1) \frac{(g (\tau_1))^{\ell -
     1}}{(\ell - 1) !} \langle \varphi (\tau_1)^{\ell n} Y^M (t_1, \ldots,
     t_k) \rangle \mathd \tau_1 + \]
  \[ - \sum_{j = 1}^k m_j \langle \omega (\tau_1) \varphi (t_j) \rangle \cdot
     \mathcal{H}^{(M - 1_j, n - 1), g}_{\ell - 1, x^n} (t_1, \ldots, t_k,
     \tau_1) . \]
  Here we use the fact that $\langle \varphi (\tau_1) \omega (\tau_1) \rangle
  = - \langle \psi (\tau_1) \bar{\psi} (\tau_1) \rangle = \frac{1}{2}$. Noting
  that
  \[ \langle \varphi^{\ell n - 2} (\tau) \psi (\tau) \bar{\psi} (\tau) Y^M
     (t_1, \ldots, t_k) \rangle + \langle \varphi (\tau) \omega (\tau) \rangle
     \langle \varphi^{\ell n - 2} (\tau) Y^M (t_1, \ldots, t_k) \rangle = 0 \]
  we obtain
  \[ \mathcal{H}^{M, g}_{\ell, x^n} (t_1, \ldots, t_k) = (- 2)^{\ell - 1}
     \left\langle Y^M (t_1, \ldots, t_k) \int_{- \infty}^{t_k} \frac{(g (\tau
     + 2 \theta \bar{\theta}))^{\ell}}{\ell !} \Phi^{n \ell} (\tau, \theta,
     \bar{\theta}) \mathd \tau \mathd \theta \mathd \bar{\theta} \right\rangle
     = \]
  \[ = \frac{(- 2)^{\ell - 1}}{\ell !} \mathcal{H}^{M, g^{\ell}}_{1, x^{n
     \ell}} (t_1, \ldots, t_k) \]
  Finally, the thesis follows from the induction hypothesis for
  $\mathcal{H}^{M, g^{\ell}}_{1, x^{n \ell}} (t_1, \ldots, t_k)$.
\end{proof}

\begin{corollary}
  \label{corollary_supersymmetry}Let $G$ be a supersymmetric functions with
  compact support then we have
  \begin{equation}
    \left\langle \varphi (0)^m \left( \int_{- \infty}^0 G (t, \theta,
    \bar{\theta}) P (\Phi (t, \theta, \bar{\theta})) \mathd t \mathd \theta
    \mathd \bar{\theta} \right)^k \right\rangle = (- 2 G_{\emptyset} (0))^k
    \langle \varphi (0)^m P (\varphi (0))^k \rangle . \label{eq:cor-13}
  \end{equation}
\end{corollary}

\begin{proof}
  Using the symmetry of the l.h.s. of~{\eqref{eq:cor-13}} with respect to the
  exchanges $(\tau_i, \theta_i, \bar{\theta}_i) \longleftrightarrow (\tau_j,
  \theta_j, \bar{\theta}_j)$ we have that
  \[ \left\langle \varphi (0)^m \left( \int_{- \infty}^0 G (t, \theta,
     \bar{\theta}) P (\Phi (\tau, \theta, \bar{\theta})) \mathd t \mathd
     \theta \mathd \bar{\theta} \right)^k \right\rangle = \]
  \[ = k! \left\langle \varphi (0)^m \int_{- \infty}^0 \int_{-
     \infty}^{\tau_1} \cdots \int_{- \infty}^{\tau_{k - 1}} \prod_{i = 1}^k G
     (\tau_i, \theta_i, \bar{\theta}_i) P (\Phi (\tau_i, \theta_i,
     \bar{\theta}_i)) \mathd \tau_i \mathd \theta_i \mathd \bar{\theta}_i
     \right\rangle . \]
  Then the claim follows directly from Lemma~\ref{lemma_superprobability}
  taking $g = G_{\emptyset}$.
\end{proof}

\begin{theorem}
  \label{theorem_localization}Let $F$ be a smooth bounded function, let $G$ be
  a supersymmetric function with compact support, let $H$ be a bounded
  function with all the derivatives bounded and let $\mathbb{K}$ be an entire
  function then we have
  \[ \left\langle F (\varphi (0)) \mathbb{K} \left( \int_{- \infty}^0 G (t,
     \theta, \bar{\theta}) H (\Phi (t, \theta, \bar{\theta})) \mathd t \mathd
     \theta \mathd \bar{\theta} \right) \right\rangle = \langle F (\varphi
     (0)) \cdot \mathbb{K} (- 2 G_{\emptyset} (0) \cdot \varphi (0)) \rangle .
  \]
\end{theorem}

\begin{proof}
  Using the density of polynomial in the set of smooth function with respect
  the topology given by the one of the Sobolev space with respect the Gaussian
  law of $\varphi (t)$, Corollary~\ref{corollary_supersymmetry} implies that
  for any $k \in \mathbb{N}$ and $F, G, H$ satisfying the hypothesis of the
  theorem
  \[ \left\langle F (\varphi (0)) \left( \int_{- \infty}^0 G (t, \theta,
     \bar{\theta}) H (\Phi (t, \theta, \bar{\theta})) \mathd t \mathd \theta
     \mathd \bar{\theta} \right)^k \right\rangle = \langle F (\varphi (0)) [-
     2 G_{\emptyset} (0) H (\varphi (0))]^k \rangle . \]
  Expanding $\mathbb{K}$ in power series and exploiting the fact that in power
  series and since
  \[ \left\langle \left| \left\langle \left( \int G (t, \theta, \bar{\theta})
     H (\Phi (t, \theta, \bar{\theta})) \mathd t \mathd \theta \mathd
     \bar{\theta} \right)^k \right\rangle_{\psi, \bar{\psi}} \right|^p
     \right\rangle \]
  is uniformly bounded when $H$ is bounded and for any $p > 1$ we can exchange
  the series with the expectation $\langle \cdot \rangle$, obtaining in this
  way the thesis.
\end{proof}

\bibliographystyle{plain}
\bibliography{langevin-supersymmetry}

\end{document}